\documentclass[11pt]{article}

\usepackage{amsmath, amssymb, amsthm, calc}  %, russian}

\title{Splitting transference inequalities \\ with the help of \\ Wolfgang Schmidt's parametric geometry of numbers.
                              \thanks{ This research was  supported by
                              RFBR (grant $\textup N^{\circ}$ 09--01--00371a) and
                              by the grant of the President of Russian Federation
                              $\textup N^\circ$ MK--1226.2010.1.
                             }}
\author{Oleg\,N.\,German}
\date{}

\theoremstyle{definition}
\newtheorem{definition}{Definition}

\theoremstyle{remark}

\theoremstyle{plain}
\newtheorem{theorem}{Theorem}
\newtheorem{lemma}{Lemma}

\newtheorem{proposition}{Proposition}
\newtheorem{corollary}{Corollary}

\newtheorem{classic}{Theorem}
\newtheorem{classicprime}[classic]{Theorem}

\renewcommand{\vec}[1]{\mathbf{#1}}
\renewcommand{\geq}{\geqslant}
\renewcommand{\leq}{\leqslant}
\renewcommand{\phi}{\varphi}

\newcommand{\R}{\mathbb{R}}
\newcommand{\Z}{\mathbb{Z}}

\newcommand{\La}{\Lambda}

\newcommand{\bPsi}{\underline{\Psi}}
\newcommand{\aPsi}{\overline{\Psi}}

\newcommand{\cB}{\mathcal{B}}

\newcommand{\cJ}{\mathcal{J}}
\newcommand{\cP}{\mathcal{P}}
\newcommand{\cQ}{\mathcal{Q}}

\newcommand{\gT}{\mathfrak{T}}
\newcommand{\ga}{\mathfrak{a}}
\newcommand{\gb}{\mathfrak{b}}

\newcommand{\tr}[1]{{#1}^\intercal}

\begin{document}

  \maketitle

  \begin{abstract}
    This paper is a sequel to our previous paper arXiv:1105.1554, where we defined two types of intermediate Diophantine exponents, connected them to Schmidt exponents and split Dyson's transference inequality into a chain of inequalities for intermediate exponents. Here we present splitting of some other transference inequalities involving both regular and uniform Diophantine exponents.
  \end{abstract}

  \section{Introduction}

  Given a matrix
  \[ \Theta=
     \begin{pmatrix}
       \theta_{11} & \cdots & \theta_{1m} \\
       \vdots & \ddots & \vdots \\
       \theta_{n1} & \cdots & \theta_{nm}
     \end{pmatrix},\qquad
     \theta_{ij}\in\R,\quad n+m\geq3, \]
  consider the system of linear equations
  \begin{equation} \label{eq:the_system}
    \Theta\vec x=\vec y
  \end{equation}
  with variables $\vec x\in\R^m$, $\vec y\in\R^n$.
  The classical measure of how well the space of solutions to this system can be approximated by integer points is defined as follows. Let $|\cdot|$ denote the sup-norm in the corresponding space.

  \begin{definition} \label{def:belpha_1}
    The supremum of the real numbers $\gamma$, such that there are arbitrarily large values of $t$ for which (resp. such that for every $t$ large enough) the system of inequalities
    \begin{equation} \label{eq:belpha_1_definition}
      |\vec x|\leq t,\qquad|\Theta\vec x-\vec y|\leq t^{-\gamma}
    \end{equation}
    has a nonzero solution in $(\vec x,\vec y)\in\Z^m\oplus\Z^n$, is called the \emph{regular} (resp. \emph{uniform}) \emph{Diophantine exponent} of $\Theta$ and is denoted by $\gb_1$ (resp. $\ga_1$).
  \end{definition}

  The \emph{transference principle} connects the problem of approximating the space of solutions to \eqref{eq:the_system} to the analogous problem for the system
  \begin{equation} \label{eq:transposed_system}
    \tr\Theta\vec y=\vec x,
  \end{equation}
  where $\tr\Theta$ denotes the transpose of $\Theta$. Let us denote the Diophantine exponents corresponding to $\tr\Theta$ by $\gb_1^\ast$ and $\ga_1^\ast$, respectively.

  The classical transference inequalities estimating $\gb_1$ in terms of $\gb_1^\ast$, and $\ga_1$ in terms of $\ga_1^\ast$ belong to A.~Ya.~Khintchine, V.~Jarn\'{\i}k, F.~Dyson and A.~Apfelbeck. In the next Section we remind some of these results and formulate their improvements obtained recently by M.~Laurent, Y.~Bugeaud and also by the author. Then, in Section \ref{sec:inter_exp}, we remind the definition of intermediate Diophantine exponents from \cite{german_inter_exp_I} and formulate the main results of this paper splitting the inequalities given in Section \ref{sec:known}. After that, in Section \ref{sec:schmidt_exp} we remind the definition of Schmidt's exponents and their relation to Diophantine exponents. Finally, in Section \ref{sec:proofs} we prove our main results.

  Notice that in \cite{german_inter_exp_I} we defined Diophantine exponents of two types. However, in this paper we shall confine our considerations to the exponents of the second type (in terminology of \cite{german_inter_exp_I}), not mentioning the exponents of the first type at all (except in this very paragraph). It is interesting whether any ``splitting'' results can be obtained for the exponents of the first type.

  \section{Known inequalities} \label{sec:known}

  \subsection{Regular exponents}

  In \cite{khintchine_palermo} A.~Ya.~Khintchine proved for $m=1$ his famous transference inequalities
  \begin{equation} \label{eq:khintchine_transference}
    \gb_1^\ast\geq n\gb_1+n-1,\qquad
    \gb_1\geq\frac{\gb_1^\ast}{(n-1)\gb_1^\ast+n}\,,
  \end{equation}
  which were generalized later by F.~Dyson \cite{dyson}, who proved that for arbitrary $n$, $m$
  \begin{equation} \label{eq:dyson_transference}
    \gb_1^\ast\geq\frac{n\gb_1+n-1}{(m-1)\gb_1+m}\,.
  \end{equation}
  While \eqref{eq:khintchine_transference} cannot be improved (see \cite{jarnik_1936_1}, \cite{jarnik_1936_2}) if only $\gb_1$ and $\gb_1^\ast$ are considered, stronger inequalities can be obtained if $\ga_1$ and $\ga_1^\ast$ are also taken into account. The corresponding result for $m=1$ belongs to M.~Laurent and Y.~Bugeaud (see \cite{laurent_2dim}, \cite{bugeaud_laurent_up_down}). They proved that if the system \eqref{eq:the_system} has no non-zero integer solutions, then
  \begin{equation} \label{eq:bugeaud_laurent}
    \frac{(\ga_1^\ast-1)\gb_1^\ast}{((n-2)\ga_1^\ast+1)\gb_1^\ast+(n-1)\ga_1^\ast}\leq
    \gb_1\leq
    \frac{(1-\ga_1)\gb_1^\ast-n+2-\ga_1}{n-1}\,.
  \end{equation}
  The inequalities \eqref{eq:bugeaud_laurent} were generalized to the case of arbitrary $n$, $m$ by the author in \cite{german_JNT}, where it was proved for arbitrary $n$, $m$ that if the space of integer solutions of \eqref{eq:the_system} is not a one-dimensional lattice, then along with \eqref{eq:dyson_transference} we have
  \begin{align}
    & \gb_1^\ast\geq
    \frac{(n-1)(1+\gb_1)-(1-\ga_1)}{(m-1)(1+\gb_1)+(1-\ga_1)}\,, \label{eq:loranoyadenie_2} \\
    & \gb_1^\ast\geq
    \frac{(n-1)(1+\gb_1^{-1})-(\ga_1^{-1}-1)}{(m-1)(1+\gb_1^{-1})+(\ga_1^{-1}-1)}
    \,, \label{eq:loranoyadenie_3}
  \end{align}
  with \eqref{eq:loranoyadenie_2} stronger than \eqref{eq:loranoyadenie_3} if and only if $\ga_1<1$.

  \subsection{Uniform exponents}

  V.~Jarn\'{\i}k and A.~Apfelbeck proved literal analogues of \eqref{eq:khintchine_transference} and \eqref{eq:dyson_transference} for the uniform exponents, i.e. with $\gb_1$, $\gb_1^\ast$ replaced by $\ga_1$, $\ga_1^\ast$, respectively (see \cite{jarnik_tiflis}, \cite{apfelbeck}). They also obtained some stronger inequalities of a more cumbersome appearance. Among them, lonely in its elegance, stands the \emph{equality}
  \begin{equation} \label{eq:jarnik_equality}
    \ga_1^{-1}+\ga_1^\ast=1
  \end{equation}
  proved by Jarn\'{\i}k for $n=1$, $m=2$. The results of Jarn\'{\i}k and Apfelbeck were improved by the author in \cite{german_JNT}, where it was proved that for arbitrary $n$, $m$ we have
  \begin{equation} \label{eq:my_inequalities_cases}
    \ga_1^\ast\geq
    \begin{cases}
      \dfrac{n-1}{m-\ga_1},\quad\ \ \text{ if }\ \ga_1\leq1, \\
      \dfrac{n-\ga_1^{-1\vphantom{\big|}}}{m-1},\quad\ \ \text{if }\ \ga_1\geq1.
    \end{cases}
  \end{equation}

  \section{Intermediate exponents} \label{sec:inter_exp}

  Set $d=m+n$. Denote by $\pmb\ell_1,\ldots,\pmb\ell_m$ the columns of the matrix
  \[ \begin{pmatrix}
       E_m\\
       \Theta
     \end{pmatrix}, \]
  where $E_m$ is the $m\times m$ unity matrix. Clearly,
%  $\cL=\spanned_\R(\pmb\ell_1,\ldots,\pmb\ell_m)$ is
  $\pmb\ell_1,\ldots,\pmb\ell_m$ span
  the space of solutions to the system \eqref{eq:the_system}. Let us set for each $k$-tuple $\sigma=\{i_1,\ldots,i_k\}$, $1\leq i_1<\ldots<i_k\leq m$,
  \begin{equation} \label{eq:L_sigma}
    \vec L_\sigma=\pmb\ell_{i_1}\wedge\ldots\wedge\pmb\ell_{i_k},
  \end{equation}
  denote by $\cJ_k$ the set of all the $k$-element subsets of $\{1,\ldots,m\}$, $k=0,\ldots,m$, and set $\vec L_\varnothing=1$.

  Let us also set $k_0=\max(0,m-p)$.

  \begin{definition} \label{def:ba}
    The supremum of the real numbers $\gamma$, such that there are arbitrarily large values of $t$ for which (resp. such that for every $t$ large enough) the system of inequalities
    \begin{equation} \label{eq:ba}
      \max_{\sigma\in\cJ_k}|\vec L_\sigma\wedge\vec Z|\leq t^{1-(k-k_0)(1+\gamma)},\qquad k=0,\ldots,m,
    \end{equation}
    has a nonzero solution in $\vec Z\in\wedge^p(\Z^d)$ is called the \emph{$p$-th regular (resp. uniform) Diophantine exponent} of $\Theta$ and is denoted by $\gb_p$ (resp. $\ga_p$).
  \end{definition}

  For $m=1$ the quantities $\gb_p$, $\ga_p$ were defined by Laurent in \cite{laurent_up_down}. The consistency of Definitions \ref{def:belpha_1} and \ref{def:ba} for arbitrary $n$, $m$ was proved in \cite{german_inter_exp_I} (see Propositions 4, 5 therein).
% !!! Â ÈÒÎÃÎÂÎÉ ÂÅÐÑÈÈ ÑÎÑËÀÒÜÑß ÍÀ ÏÐÎÏÎÇÈÖÈÞ 1 ÈËÈ ÂÎÎÁÙÅ ÍÅ ÊÎÍÊÐÅÒÈÇÈÐÎÂÀÒÜ !!! %
  Laurent and Bugeaud used the exponents $\gb_p$ to split \eqref{eq:khintchine_transference} into a chain of inequalities relating  $\gb_p$ to $\gb_{p+1}$. Namely, they proved that for $m=1$ we have  $\gb_1^\ast=\gb_n$ and
  \begin{equation} \label{eq:khintchine_transference_split}
    \gb_{p+1}\geq\frac{(n-p+1)\gb_{p}+1}{n-p}\,,\qquad
    \gb_{p}\geq\frac{p\gb_{p+1}}{\gb_{p+1}+p+1}\,,\qquad p=1,\ldots,n-1.
  \end{equation}
  Besides that, they proved for $m=1$ that if the system \eqref{eq:the_system} has no non-zero integer solutions, then we have  $\ga_1^\ast=\ga_n$ and
  \begin{equation} \label{eq:laurent_inter_mixed}
    \gb_2\geq\frac{\gb_1+\ga_1}{1-\ga_1}\,,\qquad
    \gb_{n-1}\geq\frac{1-\ga_n^{-1}}{\gb_n^{-1}+\ga_n^{-1}}\,,
  \end{equation}
  which, combined with \eqref{eq:khintchine_transference_split}, gave them \eqref{eq:bugeaud_laurent}.

  In \cite{german_inter_exp_I} we generalized \eqref{eq:khintchine_transference_split} and its analogue for the uniform exponents to the case of arbitrary $n$, $m$. We showed that
  \begin{equation} \label{eq:ab_via_ab_transposed}
    \gb_p^\ast=\gb_{d-p},\qquad\ga_p^\ast=\ga_{d-p},\qquad p=1,\ldots,d-1,
  \end{equation}
  where $\gb_p^\ast$ and $\ga_p^\ast$ are $p$-th regular and uniform Diophantine exponents of $\tr\Theta$, and proved

  \begin{theorem} \label{t:inter_dyfel}
    For each $p=1,\ldots,d-2$ the following statements hold.

    If $p\geq m$, then
    \begin{equation} \label{eq:inter_dyson_p_geq}
      (d-p-1)(1+\gb_{p+1})\geq(d-p)(1+\gb_p),
    \end{equation}
    \begin{equation} \label{eq:inter_apfel_p_geq}
      (d-p-1)(1+\ga_{p+1})\geq(d-p)(1+\ga_p).
    \end{equation}

    If $p\leq m-1$, then
    \begin{equation} \label{eq:inter_dyson_p_leq}
      (d-p-1)(1+\gb_p)^{-1}\geq(d-p)(1+\gb_{p+1})^{-1}-n,
    \end{equation}
    \begin{equation} \label{eq:inter_apfel_p_leq}
      (d-p-1)(1+\ga_p)^{-1}\geq(d-p)(1+\ga_{p+1})^{-1}-n.
    \end{equation}
  \end{theorem}

  The first result of the current paper generalizes \eqref{eq:laurent_inter_mixed}. We prove

  \begin{theorem} \label{t:inter_loranoyadenie}
    Suppose that the space of integer solutions of \eqref{eq:the_system} is not a one-dimensional lattice. Then for $m=1$ we have
    \begin{equation} \label{eq:inter_loranoyadenie_m_is_1}
      \gb_2\geq\frac{\gb_1+\ga_1}{1-\ga_1}\,,
%      \gb_2\geq\frac{2+\gb_1^{-1}-\ga_1^{-1}}{\ga_1^{-1}-1}\,.
    \end{equation}
    and for $m\geq2$ we have
    \begin{equation} \label{eq:inter_loranoyadenie_m_geq_2}
      \gb_2\geq
      \begin{cases}
        \dfrac{\ga_1-1}{2+\gb_1-\ga_1}\,,\quad\text{ if }\ \ga_1\neq\infty, \\
        \dfrac{1-\ga_1^{-1} \vphantom{\frac{\big|}{}} }{\gb_1^{-1}+\ga_1^{-1}}\,.
      \end{cases}
    \end{equation}
  \end{theorem}

  The first inequality of \eqref{eq:inter_loranoyadenie_m_is_1} is exactly the first inequality of \eqref{eq:laurent_inter_mixed}. The second inequality of \eqref{eq:inter_loranoyadenie_m_geq_2} in view of \eqref{eq:ab_via_ab_transposed} gives the second inequality of \eqref{eq:laurent_inter_mixed}.

  It follows from Theorem \ref{t:inter_dyfel} that for $m\geq2$
  \begin{equation}
    (d-2)(1+\gb_{d-1})^{-1}\leq(1+\gb_2)^{-1}+m-2.
  \end{equation}
  Combining this inequality with \eqref{eq:inter_loranoyadenie_m_geq_2} we get \eqref{eq:loranoyadenie_2} and \eqref{eq:loranoyadenie_3}, in case $m\geq2$.

  The second result of this paper splits the inequalities \eqref{eq:my_inequalities_cases}. It is the following

  \begin{theorem} \label{t:inter_my_inequalities}
    For $m=1$ we have
    \begin{equation} \label{eq:inter_my_inequalities_m_is_1}
      \ga_2\geq(1-\ga_1)^{-1}-\frac{n-2}{n-1}\,.
    \end{equation}
    For $m\geq2$ we have
    \begin{equation} \label{eq:inter_my_inequalities_m_geq_2}
      \ga_2\geq
      \begin{cases}
        \dfrac{n-1}{-n-(d-2)(1-\ga_1)^{-1}}\,,\quad\text{ if }\ \ga_1\leq1, \\
        \dfrac{m-1\vphantom{\frac{|}{}}}{\phantom{-}n+(d-2)(\ga_1-1)^{-1}}\,,\quad\text{ if }\ \ga_1\geq1.
      \end{cases}
    \end{equation}
  \end{theorem}

  Let us show that Theorem \ref{t:inter_my_inequalities} splits \eqref{eq:my_inequalities_cases} the very same way Theorem \ref{t:inter_loranoyadenie} splits \eqref{eq:loranoyadenie_2} and \eqref{eq:loranoyadenie_3}. It follows from Theorem \ref{t:inter_dyfel} that for $m=1$
  \begin{equation} \label{eq:apfel_shampur_m_is_1}
    1+\ga_n\geq(n-1)(1+\ga_2)
  \end{equation}
  and that for $m\geq2$
  \begin{equation} \label{eq:apfel_shampur_m_geq_2}
    (d-2)(1+\ga_{d-1})^{-1}\leq(1+\ga_2)^{-1}+m-2.
  \end{equation}
  Combining \eqref{eq:apfel_shampur_m_geq_2} with \eqref{eq:inter_my_inequalities_m_geq_2}, we get \eqref{eq:my_inequalities_cases} for  $m\geq2$. As for $m=1$, we always have $\ga_1\leq1$ in this case, so \eqref{eq:apfel_shampur_m_is_1} and \eqref{eq:inter_my_inequalities_m_is_1} indeed gives \eqref{eq:my_inequalities_cases} with $m=1$.

  \section{Schmidt's exponents} \label{sec:schmidt_exp}

  We start this Section with reminding the definition of Schmidt's exponents of the second type we gave in \cite{german_inter_exp_I} basing on \cite{schmidt_summerer}.

  Let $\La$ be a unimodular $d$-dimensional lattice in $\R^d$. Denote by $\cB_\infty^d$ the unit ball in sup-norm, i.e. the cube with vertices at the points $(\pm1,\ldots,\pm1)$. For each $d$-tuple $\pmb\tau=(\tau_1,\ldots,\tau_d)\in\R^d$ denote by $D_{\pmb\tau}$ the diagonal $d\times d$ matrix with $e^{\tau_1},\ldots,e^{\tau_d}$ on the main diagonal. Let us also denote by $\lambda_p(M)$ the $p$-th successive minimum of a compact symmetric convex body $M\subset\R^d$ (centered at the origin) with respect to the lattice $\La$.

  Suppose we have a path $\gT$ in $\R^d$ defined as $\pmb\tau=\pmb\tau(s)$, $s\in\R_+$, such that
  \begin{equation} \label{eq:sum_of_taus_is_zero}
    \tau_1(s)+\ldots+\tau_d(s)=0,\quad\text{ for all }s.
  \end{equation}
  In our further applications to Diophantine approximation we shall confine ourselves to a path that is a ray with the endpoint at the origin and all the functions $\tau_1(s),\ldots,\tau_d(s)$ being linear.

  Set $\cB(s)=D_{\pmb\tau(s)}\cB_\infty^d$. Consider the functions
  \[ \psi_p(\La,\gT,s)=\frac{\ln(\lambda_p(\cB(s)))}{s},\qquad p=1,\ldots,d. \]

  \begin{definition} \label{def:schmidt_Psi}
    We call the quantities
    \[ \bPsi_p(\La,\gT)=\liminf_{s\to+\infty}\bigg(\sum_{i=1}^p\psi_i(\La,\gT,s)\bigg)\,,\qquad
       \aPsi_p(\La,\gT)=\limsup_{s\to+\infty}\bigg(\sum_{i=1}^p\psi_i(\La,\gT,s)\bigg) \]
    \emph{the $p$-th lower} and \emph{upper Schmidt's exponents}, respectively.
  \end{definition}

  It appears that interpreting Diophantine exponents in terms of Schmidt's exponents simplifies many constructions and reveals the nature of some phenomena. In order to deliver this interpretation let us consider the lattice
  \begin{equation} \label{eq:La}
    \La=\Big\{ \tr{\Big(\langle\vec e_1,\vec z\rangle,\ldots,\langle\vec e_m,\vec z\rangle,\langle\pmb\ell_{m+1},\vec z\rangle,\ldots,\langle\pmb\ell_d,\vec z\rangle\Big)}\in\R^d \,\Big|\, \vec z\in\Z^d \Big\},
  \end{equation}
  where $\vec e_1,\ldots,\vec e_m$ are the first $m$ columns of the $d\times d$ unity matrix, and $\pmb\ell_{m+1},\ldots,\pmb\ell_d$ are the columns of the matrix
  \[ \begin{pmatrix}
       -\tr\Theta \\
       E_n
     \end{pmatrix}. \]
  Let us also consider the path $\gT:s\mapsto\pmb\tau(s)$ defined by
  \begin{equation} \label{eq:path}
    \tau_1(s)=\ldots=\tau_m(s)=s,\quad\tau_{m+1}(s)=\ldots=\tau_d(s)=-ms/n.
  \end{equation}

  Thus, we have connected to $\Theta$ the exponents $\bPsi_p(\La,\gT)$, $\aPsi_p(\La,\gT)$, which we shall simply denote by $\bPsi_p$ and $\aPsi_p$. We can do the same thing to $\tr\Theta$, and obtain the exponents we choose to denote by $\bPsi_p^\ast$ and $\aPsi_p^\ast$.

  In \cite{german_inter_exp_I} we proved the following statements:

  \begin{proposition} \label{prop:ba_via_Psis}
    Set $\varkappa_p=\min(p,\frac mn(d-p))$. Then
    \begin{equation} \label{eq:ba_via_Psis}
      (1+\gb_p)(\varkappa_p+\bPsi_p)=(1+\ga_p)(\varkappa_p+\aPsi_p)=d/n.
    \end{equation}
  \end{proposition}

  \begin{proposition} \label{prop:starred_ba_via_starred_Psis}
    Set $\varkappa_p^\ast=\min(p,\frac nm(d-p))$. Then
    \begin{equation} \label{eq:starred_ba_via_starred_Psis}
      (1+\gb_p^\ast)(\varkappa_p^\ast+\bPsi_p^\ast)=(1+\ga_p^\ast)(\varkappa_p^\ast+\aPsi_p^\ast)=d/m.
    \end{equation}
  \end{proposition}

  Notice that in view of \eqref{eq:ab_via_ab_transposed} it follows from \eqref{eq:ba_via_Psis}, \eqref{eq:starred_ba_via_starred_Psis} that
  \begin{equation} \label{eq:Psis_via_Psis_transposed}
    \bPsi_p^\ast=\dfrac nm\bPsi_{d-p}\quad\text{ and }\quad\aPsi_p^\ast=\dfrac nm\aPsi_{d-p}\,.
  \end{equation}

  It was also shown implicitly in \cite{german_inter_exp_I} that
  \[ \gb_p\geq\ga_p\geq\frac{d}{n\varkappa_p}-1, \]
  which is equivalent to
  \[ -\varkappa_p\leq\bPsi_p\leq\aPsi_p\leq0. \]

  Let us now translate Theorems \ref{t:inter_loranoyadenie}, \ref{t:inter_my_inequalities} into the language of Schmidt's exponents. Theorem \ref{t:inter_loranoyadenie} turns into

  \begin{theorem} \label{t:inter_loranoyadenie_Psied}
    Suppose that the space of integer solutions of \eqref{eq:the_system} is not a one-dimensional lattice. Then
    \begin{equation} \label{eq:inter_loranoyadenie_Psied}
      \bPsi_2\leq
      \begin{cases}
        2\bPsi_1+d\cdot\dfrac{\aPsi_1-\bPsi_1}{n+n\aPsi_1}\,,\quad\text{ if }\ \aPsi_1\neq-1, \\
        2\bPsi_1+d\cdot\dfrac{\aPsi_1-\bPsi_1 \vphantom{\frac{\big|}{}} }{m-n\aPsi_1}\,.
      \end{cases}
    \end{equation}
  \end{theorem}

  Theorem \ref{t:inter_my_inequalities} turns into

  \begin{theorem} \label{t:inter_my_inequalities_Psied}
    We have
    \begin{equation} \label{eq:inter_my_inequalities_Psied}
      \aPsi_2\leq
      \begin{cases}
        \dfrac{(d-2)\aPsi_1}{(n-1)+n\aPsi_1}\,,\quad\text{ if }\ \aPsi_1\geq\dfrac{m-n}{2n}\,, \\
        \dfrac{(d-2)\aPsi_1 \vphantom{\frac{\big|}{}} }{(m-1)-n\aPsi_1}\,,\quad\text{ if }\ \aPsi_1\leq\dfrac{m-n}{2n}\,.
      \end{cases}
    \end{equation}
  \end{theorem}

  As we see, this point of view relieves us of singling out the case $m=1$. In the next Section we prove Theorems \ref{t:inter_loranoyadenie_Psied}, \ref{t:inter_my_inequalities_Psied}.

  \section{Proof of Theorems \ref{t:inter_loranoyadenie_Psied}, \ref{t:inter_my_inequalities_Psied}} \label{sec:proofs}

  Having $\La$ and $\gT$ fixed by \eqref{eq:La}, \eqref{eq:path}, let us write $\psi_p(s)$ instead of $\psi_p(\La,\gT,s)$. Let us also set
  \[ \Psi_p(s)=\sum_{i=1}^p\psi_i(s). \]
  In \cite{german_inter_exp_I} we showed that
  \begin{equation} \label{eq:mink}
    0\leq-\Psi_d(s)=O(s^{-1}),
  \end{equation}
  whence we derived that for every $p$ within the range $1\leq p\leq d-2$
  \begin{equation} \label{eq:Psi_inter_dyson}
    \frac{\bPsi_{p+1}}{d-p-1}\leq\frac{\bPsi_p}{d-p}\qquad\text{ and }\qquad\frac{\aPsi_{p+1}}{d-p-1}\leq\frac{\aPsi_p}{d-p}\,,
  \end{equation}
  which is the very Theorem \ref{t:inter_dyfel} reformulated in terms of Schmidt's exponents. Now we shall need a more precise version of \eqref{eq:Psi_inter_dyson}.

  \begin{proposition} \label{prop:precise_inter_dyfel}
    For every $p$ within the range $1\leq p\leq d-2$ and every $s>0$ we have
    \begin{equation} \label{eq:Psi_sandwich}
      \frac{p+1}{p}\Psi_p(s)\leq\Psi_{p+1}(s)\leq\frac{d-p-1}{d-p}\Psi_p(s).
    \end{equation}
  \end{proposition}

  \begin{proof}
    In view of \eqref{eq:mink}, it follows from the inequalities $\psi_i(s)\leq\psi_{i+1}(s)$, $i=1,\ldots,d-1$, that
    \[ \frac 1p\sum_{i=1}^p\psi_i(s)\leq\psi_{p+1}(s)\leq\frac{-1}{d-p}\sum_{i=1}^p\psi_i(s), \]
    which immediately implies \eqref{eq:Psi_sandwich}.
  \end{proof}

  The following observation is the crucial point for proving Theorems \ref{t:inter_loranoyadenie_Psied}, \ref{t:inter_my_inequalities_Psied}.

  \begin{lemma} \label{l:main}
    Suppose $s,s'\in\R_+$ satisfy the conditions
    \begin{equation} \label{eq:main_subseteq}
      \lambda_1(\cB(s))\cB(s)\subseteq\lambda_1(\cB(s'))\cB(s'),
    \end{equation}
    \begin{equation} \label{eq:main_equals}
      \lambda_1(\cB(s'))=\lambda_2(\cB(s')).
    \end{equation}
    Then
    \begin{equation} \label{eq:main}
      \psi_2(s)\leq
      \begin{cases}
        \psi_1(s)+d\cdot\dfrac{\psi_1(s')-\psi_1(s)}{n+n\psi_1(s')}\,,\quad\text{ if }\ s'\leq s\ \text{ and }\ \psi_1(s')\neq-1, \\
        \psi_1(s)+d\cdot\dfrac{\psi_1(s')-\psi_1(s) \vphantom{\frac{\big|}{}} }{m-n\psi_1(s')}\,,\quad\text{ if }\ s'\geq s.
      \end{cases}
    \end{equation}
  \end{lemma}

  \begin{proof}
    Suppose that $s'\leq s$. Then it follows from \eqref{eq:main_subseteq} and \eqref{eq:main_equals} that 
    \[ \lambda_1(\cB(s))e^s=\lambda_1(\cB(s'))e^{s'}\geq1 \]
    and
    \[ \lambda_2(\cB(s))e^{-ms/n}\leq\lambda_2(\cB(s'))e^{-ms'/n}=\lambda_1(\cB(s'))e^{-ms'/n}, \]
    i.e.
    \begin{equation} \label{eq:s_prime_small_via_s}
      s(1+\psi_1(s))=s'(1+\psi_1(s'))\geq0
    \end{equation}
    and
    \begin{equation} \label{eq:s_prime_small_ps_2_leq_psi_1}
      s(\psi_2(s)-m/n)\leq s'(\psi_1(s')-m/n)
    \end{equation}
    Combining \eqref{eq:s_prime_small_via_s} and \eqref{eq:s_prime_small_ps_2_leq_psi_1} we get the first inequality of \eqref{eq:main}.
%    \[ \psi_2(s)\leq\frac{m+m\psi_1(s')+(1+\psi_1(s))(n\psi_1(s')-m)}{n+n\psi_1(s')} \]
%    \[ \psi_2(s)\leq\frac{d\psi_1(s')-d\psi_1(s)+n\psi_1(s)+n\psi_1(s)\psi_1(s')}{n+n\psi_1(s')} \]
    
    Suppose now that $s'\geq s$. Then it follows from \eqref{eq:main_subseteq} and \eqref{eq:main_equals} that
    \[ \lambda_1(\cB(s))e^{-ms/n}=\lambda_1(\cB(s'))e^{-ms'/n}<1 \]
    and
    \[ \lambda_2(\cB(s))e^s\leq\lambda_2(\cB(s'))e^{s'}=\lambda_1(\cB(s'))e^{s'}, \]
    i.e.
    \begin{equation} \label{eq:s_prime_large_via_s}
      s(\psi_1(s)-m/n)=s'(\psi_1(s')-m/n)<0
    \end{equation}
    and
    \begin{equation} \label{eq:s_prime_large_ps_2_leq_psi_1}
      s(1+\psi_2(s))\leq s'(1+\psi_1(s'))
    \end{equation}
    Combining \eqref{eq:s_prime_large_via_s} and \eqref{eq:s_prime_large_ps_2_leq_psi_1} we get the second inequality of \eqref{eq:main}.
  \end{proof}

  For each $\vec z=\tr{(z_1,\ldots,z_d)}\in\R^d$ and each $s>0$ let us set
  \[ \mu_s(\vec z)=e^{-s}\max_{1\leq i\leq m}|z_i|\qquad\text{ and }\qquad\nu_s(\vec z)=e^{ms/n}\max_{m<i\leq d}|z_i|. \]
  Then
  \[ \cB(s)=\Big\{ \vec z\in\R^d \,\Big|\, \mu_s(\vec z)\leq1,\ \nu_s(\vec z)\leq1 \Big\}. \]
  The parallelepiped $\lambda_1(\cB(s))\cB(s)$ contains no non-zero points of $\La$ in its interior and contains at least one pair of such points in its boundary. Of these points let us choose an arbitrary point and denote it by $\vec v_s$. Obviously, the maximal of the quantities $\mu_s(\vec v_s)$, $\nu_s(\vec v_s)$ equals $\lambda_1(\cB(s))$.

  \begin{corollary} \label{cor:for_inter_loranoyadenie}
    For each $s>0$, such that
    \begin{equation} \label{eq:for_inter_loranoyadenie_condition}
      \mu_s(\vec v_s)=\nu_s(\vec v_s)=\lambda_1(\cB(s)),
    \end{equation}
    there are $s',s''>0$, such that 
    \[ s(1+\psi_1(s))\leq s'\leq s\leq s''\leq s(1-(n/m)\psi_1(s)) \]
%    \[ \psi_1(s')\geq\psi_1(s)>-1,\qquad\psi_1(s'')\geq\psi_1(s)>-1, \]
    and
    \begin{equation} \label{eq:for_inter_loranoyadenie}
      \Psi_2(s)\leq
      \begin{cases}
        2\psi_1(s)+d\cdot\dfrac{\psi_1(s')-\psi_1(s)}{n+n\psi_1(s')}\,,\quad\text{ if }\ \psi_1(s')\neq-1, \\
        2\psi_1(s)+d\cdot\dfrac{\psi_1(s'')-\psi_1(s) \vphantom{\frac{\big|}{}} }{m-n\psi_1(s'')}\,.
      \end{cases}
    \end{equation}
  \end{corollary}

  \begin{proof}
    Let us show that the relation $\mu_s(\vec v_s)=\lambda_1(\cB(s))$ implies the existence of an $s'\leq s$ satisfying the conditions of Lemma \ref{l:main}. Denote $\lambda=\lambda_1(\cB(s))$. Let
    \[ \cP_\nu=\Big\{ \vec z\in\R^d \,\Big|\, \mu_s(\vec z)\leq\lambda,\ \nu_s(\vec z)\leq\nu\lambda \Big\} \]
    be the minimal (w.r.t inclusion) parallelepiped containing no non-zero points of $\La$ in its interior. The existence of such a parallelepiped follows from Minkowski's convex body theorem. It also implies that $1\leq\nu\leq\lambda^{-d/n}$. Then
    \[ \lambda\cB(s)\subseteq\cP_{\nu}=\lambda'\cB(s'), \]
%    \[ \overline\cP_{\nu'}=\Big\{ \vec z\in\R^d \,\Big|\, \max_{1\leq i\leq m}|z_i|\leq\lambda e^{s},\
%                                                          \max_{m<i\leq d}|z_i|\leq\nu\lambda e^{-ms/n} \Big\} \]
%    \[ \lambda'\cB(s')=\Big\{ \vec z\in\R^d \,\Big|\, \max_{1\leq i\leq m}|z_i|\leq\lambda' e^{s'},\
%                                                     \max_{m<i\leq d}|z_i|\leq\lambda' e^{-ms'/n} \Big\} \]
%    \[ \lambda e^{s}=\lambda' e^{s'},\qquad\nu e^{-ms/n}=(\lambda'/\lambda)e^{-ms'/n} \]
%    \[ e^{-ms'/n}=(\lambda'/\lambda)^{m/n}e^{-ms/n},\qquad\nu=(\lambda'/\lambda)^{d/n} \]
    where $\lambda'=\lambda\nu^{n/d}$, $s'=s-(n/d)\ln\nu$. For $\lambda'$, $s'$ we have
    \[ \lambda'\geq\lambda,\qquad s+\ln\lambda\leq s'\leq s. \]
    On the other hand, $\cP_\nu$ contains non-collinear points of $\La$ in its boundary, so $\lambda_1(\cB(s'))=\lambda_2(\cB(s'))=\lambda'$. Thus, $s$, $s'$ satisfy \eqref{eq:main_subseteq}, \eqref{eq:main_equals}.

    Now let us consider the relation $\nu_s(\vec v_s)=\lambda_1(\cB(s))$. By Minkowski's convex body theorem there is a $\mu$ in the interval $1\leq\mu\leq\lambda^{-d/m}$, such that the parallelepiped 
    \[ \cQ_\mu=\Big\{ \vec z\in\R^d \,\Big|\, \mu_s(\vec z)\leq\mu\lambda,\ \nu_s(\vec z)\leq\lambda \Big\} \]
    contains no non-zero points of $\La$ in its interior, but contains non-collinear points of $\La$ in its boundary. Then
    \[ \lambda\cB(s)\subseteq\cQ_{\mu}=\lambda''\cB(s''), \]
%    \[ \overline\cQ_{\mu}=\Big\{ \vec z\in\R^d \,\Big|\, \max_{1\leq i\leq m}|z_i|\leq\mu\lambda e^{s},\
%                                                          \max_{m<i\leq d}|z_i|\leq\lambda e^{-ms/n} \Big\} \]
%    $1\leq\mu\leq\lambda^{-d/m}$
%    \[ \lambda'\cB(s')=\Big\{ \vec z\in\R^d \,\Big|\, \max_{1\leq i\leq m}|z_i|\leq\lambda' e^{s'},\
%                                                     \max_{m<i\leq d}|z_i|\leq\lambda' e^{-ms'/n} \Big\} \]
%    \[ \mu e^{s}=(\lambda'/\lambda)e^{s'},\qquad e^{-ms/n}=(\lambda'/\lambda)e^{-ms'/n} \]
%    \[ e^{s'}=(\lambda'/\lambda)^{n/m}e^{s},\qquad\mu=(\lambda'/\lambda)^{d/m} \]
%    $\lambda'=\lambda\mu^{m/d}$, $s'=s+(n/d)\ln\mu$
    where $\lambda''=\lambda\mu^{m/d}$, $s''=s+(n/d)\ln\mu$. For $\lambda''$, $s''$ we have
    \[ \lambda''\geq\lambda,\qquad s\leq s''\leq s-(n/m)\ln\lambda. \]
    Besides that, $s$, $s''$ also satisfy \eqref{eq:main_subseteq}, \eqref{eq:main_equals}, since $\lambda_1(\cB(s''))=\lambda_2(\cB(s''))=\lambda''$.

    It remains to apply Lemma \ref{l:main}.
  \end{proof}

  Having Corollary \ref{cor:for_inter_loranoyadenie}, it is easy now to prove Theorem \ref{t:inter_loranoyadenie_Psied}. 
  
  First, let us notice that if the system \eqref{eq:the_system} has a non-zero integer solution, then it has two linearly independent integer solutions, so in this case $\aPsi_1=\bPsi_1=-1$, $\bPsi_2=-2$, which implies \eqref{eq:inter_loranoyadenie_Psied}. 
  
  Next, let us suppose that the system \eqref{eq:the_system} has no non-zero integer solutions. Then there are infinitely many local minima of $\psi_1(s)$, each of them satisfies \eqref{eq:for_inter_loranoyadenie_condition}, and the sequence of these local minima tends to $\infty$. Moreover, $s'$ and $s''$ from Corollary \ref{cor:for_inter_loranoyadenie} tend to $\infty$ as $s$ tends to $\infty$. Indeed, since \eqref{eq:the_system} has no non-zero integer solutions, we have
  \[ e^{s(1+\psi_1(s))}=e^s\lambda_1(\cB(s))=\lambda_1(e^{-s}\cB(s))\to\infty\ \ \text{ as }\ \ s\to\infty, \]
  so
  \begin{equation} \label{eq:tending_to_infty}
    s(1+\psi_1(s))\to\infty\ \ \text{ as }\ \ s\to\infty.
  \end{equation}
  Particularly, it follows from \eqref{eq:tending_to_infty} that $\psi_1(s)$ is eventually greater than $-1$ (it can actually be shown that $\psi_1(s)>-1$ starting with the second local minimum of $\psi_1(s)$).
  Therefore,
  \begin{equation} \label{eq:for_inter_loranoyadenie_liminfsup}
    \bPsi_2\leq\liminf\Psi_2(s)\leq
    \begin{cases}
      2\liminf\psi_1(s)+d\cdot\limsup\dfrac{\psi_1(s')-\psi_1(s)}{n+n\psi_1(s')}\,, \\
      2\liminf\psi_1(s)+d\cdot\limsup\dfrac{\psi_1(s'')-\psi_1(s) \vphantom{\frac{\big|}{}} }{m-n\psi_1(s'')}\,,
    \end{cases}
  \end{equation}
  where the $\liminf$ and the $\limsup$ are taken over the set of local minima of $\psi_1(s)$. Since $\psi_1(s)$ is never positive, both denominators in \eqref{eq:for_inter_loranoyadenie_liminfsup} are eventually positive. Therefore, \eqref{eq:for_inter_loranoyadenie_liminfsup} implies \eqref{eq:inter_loranoyadenie_Psied}.

  \begin{corollary} \label{cor:for_inter_my_inequalities}
    Suppose that the system \eqref{eq:the_system} has no non-zero integer solutions. Then for each $s>0$ there is an $s'>0$, such that $s(1+\psi_1(s))\leq s'\leq s(1-(n/m)\psi_1(s))$, and
    \begin{equation} \label{eq:for_inter_my_inequalities}
      \Psi_2(s)\leq
      \begin{cases}
        \dfrac{(d-2)\psi_1(s')}{(n-1)+n\psi_1(s')}\,,\quad\text{ if }\ \psi_1(s')\geq\dfrac{m-n}{2n}\,, \\
        \dfrac{(d-2)\psi_1(s') \vphantom{\frac{\big|}{}} }{(m-1)-n\psi_1(s')}\,,\quad\text{ if }\ \psi_1(s')\leq\dfrac{m-n}{2n}\,.
      \end{cases}
    \end{equation}
  \end{corollary}
  
  \begin{proof}
    Assume that $\mu_s(\vec v_s)=\lambda_1(\cB(s))$. Then the same argument as in the proof of Corollary \ref{cor:for_inter_loranoyadenie} shows that there is an $s'$, such that $s(1+\psi_1(s))\leq s'\leq s$, and
    \begin{equation} \label{eq:corollary_n}
      \Psi_2(s)\leq2\psi_1(s)+d\cdot\dfrac{\psi_1(s')-\psi_1(s)}{n+n\psi_1(s')}\,,
    \end{equation}
    unless $\psi_1(s')=-1$.
    By Proposition \ref{prop:precise_inter_dyfel} we have
    \begin{equation} \label{eq:psi_sandwiched_by_Psi}
      \frac{d-1}{d-2}\Psi_2(s)\leq\psi_1(s)\leq\frac 12\Psi_2(s).
    \end{equation}
    If $\psi_1(s')=-1$, then \eqref{eq:psi_sandwiched_by_Psi} implies \eqref{eq:for_inter_my_inequalities}. Suppose that $\psi_1(s')\neq-1$. Then, taking into account that
    \begin{equation*} %\label{eq:positiveness_n}
      2-\frac{d}{n+n\psi_1(s')}\geq0\quad\text{ if and only if }\quad\psi_1(s')\geq\frac{m-n}{2n}\,,
    \end{equation*}
    we conclude from \eqref{eq:corollary_n} and \eqref{eq:psi_sandwiched_by_Psi} that
    \begin{equation} \label{eq:if_s_prime_is_less_than_s}
      \Psi_2(s)\leq
      \begin{cases}
        \qquad\ 2\psi_1(s')\,,\qquad\quad\,\ \text{ if }\ \psi_1(s')\geq\dfrac{m-n}{2n}\,, \\
        \dfrac{(d-2)\psi_1(s')}{(m-1)-n\psi_1(s')}\,,\quad\text{ if }\ \psi_1(s')\leq\dfrac{m-n \vphantom{\frac{\big|}{}} }{2n}\,.
      \end{cases}
    \end{equation}
    
    Assume now that $\nu_s(\vec v_s)=\lambda_1(\cB(s))$. Then the same argument as in the proof of Corollary \ref{cor:for_inter_loranoyadenie} shows that there is an $s''$, such that $s\leq s''\leq s(1-(n/m)\psi_1(s))$, and
    \begin{equation} \label{eq:corollary_m}
      \Psi_2(s)\leq2\psi_1(s)+d\cdot\dfrac{\psi_1(s'')-\psi_1(s)}{m-n\psi_1(s'')}\,.
    \end{equation}
    Taking into account that
    \begin{equation*} %\label{eq:positiveness_m}
      2-\frac{d}{m-n\psi_1(s'')}\geq0\quad\text{ if and only if }\quad\psi_1(s'')\leq\frac{m-n}{2n}\,,
    \end{equation*}
    we conclude from \eqref{eq:corollary_m} and \eqref{eq:psi_sandwiched_by_Psi} that
    \begin{equation} \label{eq:if_s_prime_is_greater_than_s}
      \Psi_2(s)\leq
      \begin{cases}
        \dfrac{(d-2)\psi_1(s'')}{(n-1)+n\psi_1(s'')}\,,\quad\text{ if }\ \psi_1(s'')\geq\dfrac{m-n}{2n}\,, \\
        \qquad\ 2\psi_1(s'')\,,\qquad\quad\ \text{ if }\ \psi_1(s'')\leq\dfrac{m-n \vphantom{\frac{\big|}{}} }{2n}\,.
      \end{cases}
    \end{equation}
    
    Since $\psi_1(s')$ and $\psi_1(s'')$ are negative, we have
    \[ 2\psi_1(s')\leq\dfrac{(d-2)\psi_1(s')}{(n-1)+n\psi_1(s')}\,,\qquad\text{ if }\ \psi_1(s')\geq\dfrac{m-n}{2n}\,, \]
    and
    \[ 2\psi_1(s'')\leq\dfrac{(d-2)\psi_1(s'')}{(m-1)-n\psi_1(s'')}\,,\quad\text{ if }\ \psi_1(s'')\leq\dfrac{m-n}{2n}\,. \]
    Therefore, \eqref{eq:if_s_prime_is_less_than_s} and \eqref{eq:if_s_prime_is_greater_than_s} imply the desired statement.
  \end{proof}
  
  Deriving Theorem \ref{t:inter_my_inequalities_Psied} from Corollary \ref{cor:for_inter_my_inequalities} is even easier than deriving Theorem \ref{t:inter_loranoyadenie_Psied} from Corollary \ref{cor:for_inter_loranoyadenie}.
  
  If the system \eqref{eq:the_system} has a non-zero integer solution, then $\aPsi_1=-1<\frac{m-n}{2n}$\,, and \eqref{eq:inter_my_inequalities_Psied} follows from \eqref{eq:Psi_inter_dyson}. Suppose now that \eqref{eq:the_system} has no non-zero integer solutions. Then it follows from \eqref{eq:tending_to_infty} that $s'$ from Corollary \ref{cor:for_inter_my_inequalities} tends to $\infty$ as $s$ tends to $\infty$. Hence, taking $\limsup$ of both sides in \eqref{eq:for_inter_my_inequalities}, we get \eqref{eq:inter_my_inequalities_Psied}.

\vskip 10mm

\noindent
Oleg N. {\sc German} \\
Moscow Lomonosov State University \\
Vorobiovy Gory, GSP--1 \\
119991 Moscow, RUSSIA \\
\emph{E-mail}: {\fontfamily{cmtt}\selectfont german@mech.math.msu.su, german.oleg@gmail.com}

\end{document}